\documentclass[twoside,12pt]{article}
\usepackage{graphicx,amsmath,latexsym,amssymb,amsthm}
\usepackage[colorlinks=black,urlcolor=black]{hyperref}
at 9pt  

\date{}
\newtheorem{theorem}{Theorem}[section]

\newtheorem{thm}{Theorem}[section]

\newtheorem{cor}[thm]{Corollary}

\numberwithin{equation}{section}

\begin{document}
\setlength{\unitlength}{1cm}



\vskip1.5cm

\centerline { \textbf{ Completeness of
 one two-interval boundary   }}
 \centerline { \textbf{value problem  with transmission conditions  }}

\vskip.2cm


\vskip.8cm \centerline {\textbf{ K. Aydemir$^\dag$ and  O. Sh.
Mukhtarov$^\dag$}}

\vskip.5cm

\centerline {$^\dag$Department of Mathematics, Faculty of Science,}
\centerline {Gaziosmanpa\c{s}a University,
 60250 Tokat, Turkey}
\centerline {e-mail : {\tt kadriye.aydemir@gop.edu.tr,
omukhtarov@yahoo.com }}




\vskip.5cm \hskip-.5cm{\small{\bf Abstract :} This paper presents a
new approach to the two-interval Sturm-Liouville eigenfunction
expansions, based essentially on the method of integral equations.
We consider the Sturm-Liouville problem together with two
supplementary transmission conditions at one interior point. Further
we develop Green's function method for spectral analysis of the
considered problem in modified Hilbert space. \vskip0.3cm\noindent
{\bf Keywords :} \
 Sturm-Liouville problems,
transmission conditions, expansions theorem.{\vskip0.3cm\noindent

\hrulefill

\section{\textbf{Introduction}}

The development of classical, rather then the operatoric,
Sturm-Liouville theory in the years after 1950 can be found in
various sources; in particular in the texts of Atkinson \cite{at},
Coddington and Levinson \cite{co}, Levitan and Sargsjan \cite{le}.
The operator theoretic development is given in the texts by Naimark
\cite{na} and Akhiezer and Glazman \cite{ak}. The subject of
eigenfunction expansions is as old as operator theory. The
Sturm-Liouville problem is also important because the solutions to a
homogeneous BVP with  homogeneous BCs produce a set of orthogonal
functions. Such functions can be used to represent functions in
Fourier series expansions. The completeness of classical systems of
eigenfunction expansions was originally related to mechanical
problems and boundary value problems for differential operators.
Later the study of eigenfunctions expansions has gained an
independent and abstract status. The expansion has an integral
operator form whose kernel is a spectral function, the
representation of which is the Green function of the operator. For
the method of treating such problems see \cite{le1}. The method of
Sturm expansions is widely used in calculations of the spectroscopic
characteristics of atoms and molecules \cite{ro, she}. The
advantages of the method are both the possibility of adequate choice
of the nature of the spectrum of an unperturbed problem in the
calculation of quantum mechanical systems using the perturbation
theory and the possibility of providing best convergence of
corresponding expansions in basis functions. In this study we shall
investigate one two-interval boundary value  problem ,namely, the
Sturm-Liouville equation
\begin{eqnarray}\label{1.1}
-p(x)y^{\prime \prime }(x)+ q(x)y(x)=\lambda y(x)
\end{eqnarray}
to hold in disjoint intervals $(-\pi,0) \ \textrm{and} \ (0,\pi) $
 under boundary conditions
\begin{eqnarray}\label{1.2}
\cos \alpha y(-\pi)+\sin\alpha y'(-\pi)=0
\end{eqnarray}
\begin{eqnarray}\label{1.3}
\cos\beta y(\pi)+\sin\beta y'(\pi)=0.
\end{eqnarray}
where singularity of the solution $y=y(x,\lambda)$ prescribed by
transmission conditions
\begin{eqnarray}\label{1.4}
\beta^{-}_{11}y'(0-)+\beta^{-}_{10}y(0-)+\beta^{+}_{11}y'(0+)+\beta^{+}_{10}y(0+)=0
\end{eqnarray}
\begin{eqnarray}\label{1.5}
\beta^{-}_{21}y'(0-)+\beta^{-}_{20}y(0-)+\beta^{+}_{21}y'(0+)+\beta^{+}_{20}y(0+)=0
\end{eqnarray}
where $p(x)=p_1>0 \ \textrm{for} \ x \in [-\pi, 0) $, $p(x)=p_2>0 \
\textrm{for} \ x \in (0, \pi], $ the potential $q(x)$ is given
real-valued function which continuous in each of the intervals
$[-\pi, 0) \ \textrm{and} \ (0, \pi]$, and has a finite limits
$q(\mp0)$, $\lambda$ \ is a complex eigenparameter, \
$\beta^{\pm}_{ij},\ (i=1,2 \ \textrm{and} \ j=0,1)$ are real
numbers. In this study we introduce a new approach to the
two-interval Sturm-Liouville eigenfunction expansions, based
essentially on the method of integral equations. Note that in
physics many problems arise in the form of boundary value problems
involving second order ordinary differential equations. This
derivation based on that in \cite{kr}, however, is more thorough
than that in most elementary physics texts; while most parameters
such as density and other thermal properties are treated as constant
in such treatments, the following allows fundamental properties of
the bar to vary as a function of the bar's length, which will lead
to a Sturm-Liouville problem of a more general nature.

\section{Construction of the resolvent by means of Green' s function }
\subsection{The  resolvent and the Green' s
function } Let $u=\phi _{1}(x,\lambda )$ be solution of the equation
(\ref{1.1}) on left interval $[-\pi,0)$ satisfying the conditions
\begin{equation}\label{tam1}
 \ \ u(-\pi)=\sin\alpha , \
\ u'(-\pi)=-\cos\alpha
\end{equation}
and $u=\phi _{2}(x,\lambda )$ be solution of same equation on right
interval $(0,\pi]$ satisfying the conditions
\begin{eqnarray}\label{tamm1}
&&u(0+) =\frac{1}{\Delta_{12}}(\Delta_{23}\phi_{1}(0,\lambda
)+\Delta_{24}\frac{\partial\phi_{1}(0,\lambda )}{\partial x})
\\ &&\label{tamm2}u'(0+)
=\frac{-1}{\Delta_{12}}(\Delta_{13}\phi_{1}(0,\lambda
)+\Delta_{14}\frac{\partial\phi_{1}(0,\lambda )}{\partial x})
\end{eqnarray}
Similarly we define by $u=\chi _{2}(x,\lambda )$ the solution of the
equation (\ref{1.1}) on the right interval $(0,\pi]$ by initial
conditions
\begin{equation}\label{4}
 \ \ u(\pi)=-\sin\beta, \ \ \ \
\  u'(\pi)=\cos\beta
\end{equation}
 and by $u=\chi _{1}(x,\lambda )$ the solution of the equation
 (\ref{1.1}) on the left interval $[-\pi,0)$ by initial
conditions
\begin{eqnarray}\label{6}
&&u(0-) =\frac{-1}{\Delta_{34}}(\Delta_{14}\chi_{2}(0,\lambda
)+\Delta_{24}\frac{\partial\chi_{2}(0,\lambda )}{\partial x})
\\ &&\label{6t}  u'(0-,\lambda )
=\frac{1}{\Delta_{34}}(\Delta_{13}\chi_{2}(0,\lambda
)+\Delta_{23}\frac{\partial\chi_{2}(0,\lambda )}{\partial x})
\end{eqnarray}
respectively,  where $\Delta_{ij} \ (1\leq i< j \leq 4)$ denotes the
determinant of the i-th
 and j-th columns of the matrix $$ T= \left[%
\begin{array}{cccc}
  \beta^{+}_{10} & \beta^{+}_{11} & \beta^{-}_{10} & \beta^{-}_{11}\\
\beta^{+}_{20} & \beta^{+}_{21} &   \beta^{-}_{20} & \beta^{-}_{21}
  \\
\end{array} %
 \right]. $$
By applying the method of \cite{ka}  we can show that each of these
solutions are entire functions of complex parameter $ \lambda.$ It
is clear that each of the Wronskians
\begin{equation}\label{wr}
\omega_i(\lambda)=W(\phi _{i}(x,\lambda ), \chi _{i}(x,\lambda )), \
\ Ý=1,2
\end{equation}
are independent of variable x. By using the definitions of solutions
$\phi _{i}(x,\lambda )$ and $\chi _{i}(x,\lambda )$ we have
\begin{eqnarray} \label{wr1}
w_{2}(\lambda) &=&\phi_{2}(0+,\lambda
)\frac{\partial\chi_{2}(0+,\lambda )}{\partial
x}-\frac{\partial\phi_{2}(0+,\lambda )}{\partial x}\chi_{2}(0+,\lambda )\nonumber \\
 &=&\frac{\Delta_{34}}{\Delta_{12}}(\phi_{1}(0-,\lambda )\frac{\partial\chi_{1}
 (0-,\lambda )}{\partial x}-\frac{\partial\phi_{1}(0-,\lambda )}{\partial x}\chi_{1}(0-,\lambda )) \nonumber \\
&=&\frac{\Delta_{34}}{\Delta_{12}}w_{1}(\lambda)
\end{eqnarray}
Introduce to the consideration the Green's function (see, \cite{ay})
$G(x,\xi,\lambda)$ of the problem $(\ref{1.1})-(\ref{1.5})$ as
\begin{eqnarray}\label{gr}
G(x,\xi;\lambda)=\left\{\begin{array}{c}
\frac{\phi_{1}(x,\lambda)\chi_{1}(\xi,\lambda)}{\Delta_{34}p_1\omega_{1}(\lambda)}, \ \ \ if \ x \in[-\pi,0), \  \xi \in [-\pi,x)\ \  \\
                \\
\frac{\chi_{1}(x,\lambda)\phi_{1}(\xi,\lambda)}{\Delta_{34}p_1\omega_{1}(\lambda)},  \ \ \ if \ x \in[-\pi,0), \  \xi \in [x,0)\ \  \\
  \\
\frac{\chi_{1}(x,\lambda)\phi_{2}(\xi,\lambda)}{\Delta_{34}p_1\omega_{1}(\lambda)}, \ \ \ \ \ if \ x \in[-\pi,0), \  \xi \in (0,\pi]\ \    \\
                \\
\frac{\phi_{2}(x,\lambda)\chi_{2}(\xi,\lambda)}{\Delta_{12}p_2\omega_{2}(\lambda)}, \ \ \ if \ \ x \in(0,\pi], \ \xi \in (0,x]\ \   \\
  \\
 \frac{\chi_{2}(x,\lambda) \phi_{2}(\xi,\lambda)}{\Delta_{12}p_2\omega_{2}(\lambda)}, \ \ \ if \ \ x \in(0,\pi], \ \xi \in [x,\pi]\ \   \\
\\
\frac{\phi_{2}(x,\lambda)\chi_{1}(\xi,\lambda)}{\Delta_{12}p_2\omega_{2}(\lambda)}, \ \ \ if \ x \in(0,\pi], \ \xi \in [-\pi,0)\ \   \\
 \\
             \end{array}\right.
\end{eqnarray}
Let $f(x)$ be any function continuous in  $[-\pi, 0) \ \textrm{and}
\ (0, \pi]$,  which has finite limits $f(\mp 0)$. Show that the
function
\begin{eqnarray}\label{gr1}
u(x,\lambda)&=& \frac{\Delta_{34}}{p_1}\int_{-\pi}^{0}
G(x,\xi;\lambda)f(\xi)d\xi+
 \frac{\Delta_{12}}{p_2}\int_{0}^{\pi}G(x,\xi;\lambda)f(\xi)d\xi
\end{eqnarray}
satisfies the nonhomogeneous equation
\begin{eqnarray}\label{gr2}
-p(x)u''+\{q(x)-\lambda\}u=f(x), \ x \in [-\pi,0)\cup(0,\pi]
\end{eqnarray}
and boundary-transmission conditions (BTC) (\ref{1.2})-(\ref{1.5}).
Indeed, putting (\ref{gr}) in (\ref{gr1}) and then  differentiating
twice we have
\begin{eqnarray}\label{gr3}
u''(x,\lambda)&=&\left\{\begin{array}{c}
\frac{\chi''_{1}(x,\lambda)}{p_1\omega_1(\lambda)}\int_{-\pi}^{x}\phi_{1}(\xi,\lambda)f(\xi)d\xi
+
\frac{\phi''_{1}(x,\lambda)}{p_1\omega_1(\lambda)}\int_{x}^{0}\chi_{1}(\xi,\lambda)f(\xi)d\xi \\ +\frac{\phi''_{1}(x,\lambda)}{p_2\omega_2(\lambda)}\int_{0}^{\pi}\chi_{2}(\xi,\lambda)f(\xi)d\xi-\frac{f(x)p_1}{p_1\omega_1(\lambda)} W(\phi_{1}(x,\lambda),\chi_{1}(x,\lambda)) \\ \ \ \ \ \ \ \ \ for \  x \in [-\pi,0) \\
                \\
\frac{\chi''_{2}(x,\lambda)}{p_1\omega_1(\lambda)}\int_{-\pi}^{0}\phi_{1}(\xi,\lambda)f(\xi)d\xi+\frac{\chi''_{2}(x,\lambda)}{\omega_2(\lambda)}\int_{0}^{x}\phi_{2}(\xi,\lambda)f(\xi)d\xi\\
+
\frac{\phi''_{2}(x,\lambda)}{p_2\omega_2(\lambda)}\int_{x}^{\pi}\chi_{2}(\xi,\lambda)f(\xi)d\xi-\frac{f(x)p_2}{p_2\omega_2(\lambda)} W(\phi_{2}(x,\lambda),\chi_{2}(x,\lambda))\  \\ \ \ \ \ \ \ \ for \  x \in (0,\pi] \\
             \end{array}\right.
\\ \label{gr4} &=&(q(x)-\lambda)\left\{\begin{array}{c}
\frac{\chi_{1}(x,\lambda)}{p_1\omega_1(\lambda)}\int_{-\pi}^{x}\phi_{1}(\xi,\lambda)f(\xi)d\xi
+
\frac{\phi_{1}(x,\lambda)}{p_1\omega_1(\lambda)}\int_{x}^{0}\chi_{1}(\xi,\lambda)f(\xi)d\xi \\ +\frac{\phi_{1}(x,\lambda)}{p_2\omega_2(\lambda)}\int_{0}^{\pi}\chi_{2}(\xi,\lambda)f(\xi)d\xi-f(x) \ \ \ \ \ \ \ \ for \  x \in [-\pi,0) \\
                \\
\frac{\chi_{2}(x,\lambda)}{p_1\omega_1(\lambda)}\int_{-\pi}^{0}\phi_{1}(\xi,\lambda)f(\xi)d\xi+\frac{\phi_{2}(x,\lambda)}{p_2\omega_2(\lambda)}\int_{x}^{\pi}\chi_{2}(\xi,\lambda)f(\xi)d\xi\\
+
\frac{\chi_{2}(x,\lambda)}{p_2\omega_2(\lambda)}\int_{0}^{x}\phi_{2}(\xi,\lambda)f(\xi)d\xi-f(x) \ \ \ \ \ \ \ for \  x \in (0,\pi] \\
             \end{array}\right.
\nonumber\\ \nonumber\\ &=&(q(x)-\lambda)u(x,\lambda)-f(x)
\end{eqnarray}
i.e. the function $u(x,\lambda)$ satisfies the  nonhomogeneous
equation (\ref{gr2}).  The function $u(x,\lambda)$ also satisfies
the transmission conditions (\ref{1.4})-(\ref{1.5})
 and both boundary conditions (\ref{1.2})-(\ref{1.3}).
Therefore $u(x,\lambda)$ forms an resolvent of the problem
(\ref{1.1})-(\ref{1.5}).
\section{Expansion results for Green's function}
Through in below we assume that the homogeneous equation
\begin{eqnarray}\label{ex1}
p(x)u''-q(x)u=0
\end{eqnarray}
under the same BTC's (\ref{1.2})-(\ref{1.5}) has only the trivial
solution $u=0$. This amounts to the assumption that $\lambda=0$ is
not an eigenvalue of the considered BVTP (\ref{1.1})-(\ref{1.5}).
There is no less of generality in this assumption, since otherwise
we may consider a new equation
$$-p(x)u''+\widetilde{q}(x)u=\lambda u$$
for $\widetilde{q}(x)=q(x)-\widetilde{\lambda}$ under the same BTC's
(\ref{1.2})-(\ref{1.5}). Obviously this problem has the same
eigenfunctions as for the considered BVTP (\ref{1.1})-(\ref{1.5}),
all  eigenvalues are shifted through $\widetilde{\lambda}$ to the
right and therefore, we can chose $\widetilde{\lambda}$  such that
$\lambda=0$ is not eigenvalue of the new problem.  Now defining
$G_{0}(x,\xi)=G(x,\xi;0)$ we see that the function
 \begin{eqnarray}\label{ex2}
u_{0}(x)&=& \frac{\Delta_{34}}{p_1}\int_{-\pi}^{0}
G_{0}(x,\xi)f(\xi)d\xi + \frac{\Delta_{12}}{p_2} \int_{0}^{\pi}
G_{0}(x,\xi)f(\xi)d\xi
\end{eqnarray}
solves the nonhomogeneous equation $p(x)u''-q(x)u=f(x)$ and
satisfies all boundary-transmission conditions
(\ref{1.2})-(\ref{1.5}). Rewritting (\ref{gr2}) in the form
\begin{eqnarray}\label{ex3}
-p(x)u''-q(x)u=\lambda u-f(x)
\end{eqnarray}
We see that this equation under the same BTC's
(\ref{1.2})-(\ref{1.5})  reduces to the integral equation
\begin{eqnarray}\label{ex4}
&&u(x)+\lambda \{\frac{\Delta_{34} }{p_1}\int_{-\pi}^{0}
G_{0}(x,\xi)u(\xi)d\xi+ \frac{\Delta_{12}}{p_2}\int_{0}^{\pi}
G_0(x,\xi)u(\xi)d\xi \} \nonumber\\&=&
\frac{\Delta_{34}}{p_1}\int_{-\pi}^{0} G_0(x,\xi)f(\xi)d\xi+
\frac{\Delta_{12} }{p_2}\int_{0}^{\pi} G_0(x,\xi)f(\xi)d\xi
\end{eqnarray}
Consequently the considered BVTP (\ref{1.1})-(\ref{1.5}) can be
written in equivalent integral equation  form  as
\begin{eqnarray}\label{ex5}
u(x)=-\lambda \{\frac{\Delta_{34} }{p_1}\int_{-\pi}^{0}
G_0(x,\xi)u(\xi)d\xi+\frac{ \Delta_{12}}{p_2} \int_{0}^{\pi}
G_0(x,\xi)u(\xi)d\xi \}
\end{eqnarray}
By applying the same method as in \cite{ka} we can prove that the
BVTP (\ref{1.1})-(\ref{1.5}) has precisely denumerable many
eigenvalues $\lambda_{0}, \lambda_{1}, \lambda_{2},...,$ with
following asymptotic behaviour as $n\rightarrow \infty. $\\
\textbf{(i)} If $\sin\beta \neq 0$ and $\sin\alpha\neq 0$, then
\begin{equation}
s_{n}= (\frac{n-1}{2})+O\left( \frac{1}{n}\right) \label{(5.1)}
\end{equation}
\textbf{(ii)} If $\sin\beta\neq  0$ and $\sin\alpha= 0$, then
\begin{equation}
s_{n}=\frac{n}{2} +O\left( \frac{1}{n}\right), \label{(5.2)}
\end{equation}
\textbf{(iii)} If $\sin\beta= 0$ and $\sin\alpha\neq 0$, then
\begin{equation}
s_{n}=\frac{n}{2}+O\left( \frac{1}{n}\right) , \label{(5.3)}
\end{equation}
\textbf{(iv)} If $\sin\beta=0$ and $\sin\alpha= 0$, then
\begin{equation}
s_{n}=\frac{n}{2} +O\left( \frac{1}{n}\right) , \label{(5.4)}
\end{equation}
 as $n\rightarrow \infty$ where $\lambda _{n}=s_{n}^{2}$ $\label{t4}$. It is evident that the
functions $\phi_n(x)=\phi(x,\lambda_n)$ are eigenfunctions
corresponding to the eigenvalues $\lambda_n$. Let
\begin{eqnarray}\label{ex8}
\varphi_n(x)=(\frac{\Delta_{34}}{p_1}\int_{-\pi}^{0}\phi_n^{2}(\xi)d\xi+
\frac{\Delta_{12}}{p_2} \int_{0}^{\pi}
\phi_n^{2}(x)dx)^{-\frac{1}{2}}\Phi_n(x)
\end{eqnarray}
It is obviously seen that the sequence $(\varphi_n(x))$ forms an
orthonormal set of eigenfunctions in the sense of
\begin{eqnarray}\label{ex9}
\frac{\Delta_{34}}{p_1}\int_{-\pi}^{0}\varphi_{n}(x)\varphi_{m}(x)dx+
\frac{\Delta_{12} }{p_2}\int_{0}^{\pi}
\varphi_{n}(x)\varphi_{m}(x)dx=\delta_{nm},
\end{eqnarray}
where $\delta_{nm}$ is the kronecker delta.
\begin{theorem}\label{ext}
The Green's function $G_0(x,\xi)$ can be expanded into an
eigenfunction series
\begin{eqnarray}\label{ex10}
G_0(x,\xi)=-\sum_{n=0}^{\infty}\lambda_{n}^{-1}
\varphi_{n}(x)\varphi_{n}(\xi)
\end{eqnarray}
which converges absolutely and uniformly on
$([-\pi,0)\cup(0,\pi])^{2}.$
\end{theorem}
\begin{proof}
By using the definition of eigenfunctions $\varphi_{n}(x)$ and
asymptotic behaviour of eigenvalues it is not difficult to show that
the series in (\ref{ex10}) converges absolutely and uniformly and
therefore represents a continuous function there. To prove the
equality (\ref{ex10}), suppose, it possible that, the function
\begin{eqnarray}\label{ex11}\widetilde{G_0}(x,\xi)=G_0(x,\xi)+\sum_{n=0}^{\infty} \lambda_{n}^{-1} \varphi_{n}(x)\varphi_{n}(\xi)
\end{eqnarray}
is not identically zero. Taking in view that the Kernel
$\widetilde{G_0}(x,\xi)$ is symmetric slightly modifying the method
of proving of familiar theorem in the theory of integral equations
(see, for example, \cite{ko}) which assert that any symmetric kernel
which is not identically zero has at least one eigenfunction, we can
prove that there is a real number $\mu_0\neq0$ and real valued
function $\psi_0\neq0$ such that
\begin{eqnarray}\label{ex12}
\frac{\Delta_{34}}{p_1}\int_{-\pi}^{0}
\widetilde{G}_0(x,\xi)\psi_{0}(\xi)d\xi+
\frac{\Delta_{12}}{p_2}\int_{0}^{\pi}
\widetilde{G}_0(x,\xi)\psi_{0}(\xi)d\xi=\mu_0\psi_0(x)
\end{eqnarray}
Multiplying by $\varphi_{n}(x)$ and make the necessary calculations
we have
\begin{eqnarray}\label{ex13} && \mu_0(\frac{\Delta_{34}}{p_1}\int_{-\pi}^{0}\psi_0(x)\varphi_{n}(x)dx+
\frac{\Delta_{12}}{p_2} \int_{0}^{\pi}
\psi_0(x)\varphi_{n}(x)dx)=\nonumber\\
&&\frac{\Delta_{34}}{p_1}
\int_{-\pi}^{0}\psi_0(\xi)(\frac{\Delta_{34}}{p_1}\int_{-\pi}^{0}
\widetilde{G_0}(x,\xi)\varphi_{n}(x)dx+ \frac{\Delta_{12}}{p_2}
\int_{0}^{\pi} \widetilde{G_0}(x,\xi)\varphi_{n}(x)dx)d\xi\nonumber\\
&+&\frac{\Delta_{12}}{p_2}\int_{0}^{\pi}\psi_0(\xi)(\frac{\Delta_{34}}{p_1}\int_{-\pi}^{0}
\widetilde{G_0}(x,\xi)\varphi_{n}(x)dx+ \frac{\Delta_{12}}{p_2}
\int_{0}^{\pi} \widetilde{G_0}(x,\xi)\varphi_{n}(x)dx)d\xi
\end{eqnarray}
Recalling that the set of eigenfunctions $(\varphi_{n}(x))$ is
orthonormal in the sense of (\ref{ex9}) it is easy to see that
\begin{eqnarray}\label{ex14} && \frac{\Delta_{34}}{p_1}\int_{-\pi}^{0}
\widetilde{G_0}(x,\xi)\varphi_{n}(x)dx+ \frac{\Delta_{12}}{p_2}
\int_{0}^{\pi} \widetilde{G_0}(x,\xi)\varphi_{n}(x)dx \nonumber\\
&=& \frac{\Delta_{34}}{p_1}\int_{-\pi}^{0} G(x,\xi)\varphi_{n}(x)dx+
\frac{\Delta_{12}}{p_2}\int_{0}^{\pi}
G(x,\xi)\varphi_{n}(x)dx+\frac{\varphi_{n}(\xi)}{\lambda_n}
\end{eqnarray}
Substituting (\ref{ex14}) in (\ref{ex13})  and  taking in view the
fact that the eigenfunction $\varphi_{n}(x)$ satisfy the integral
equation (\ref{ex5}) for $\lambda=\lambda_n$ gives
\begin{eqnarray}\label{ex15} &&\frac{ \Delta_{34}}{p_1}\int_{-\pi}^{0}\psi_0(x)\varphi_{n}(x)dx+
\frac{\Delta_{12}}{p_2}\int_{0}^{\pi} \psi_0(x)\varphi_{n}(x)dx=0
\end{eqnarray}
for all $n=0,1,2...$, i.e. the function $\psi_0(x)$  is orthogonal
to all eigenfunctions. On the other hand, from (\ref{ex12}) and
(\ref{ex15}) it follows that
\begin{eqnarray}\label{ex16}
\frac{\Delta_{34}}{p_1}\int_{-\pi}^{0} G_0(x,\xi)\psi_{0}(\xi)d\xi+
\frac{\Delta_{12}}{p_2}\int_{0}^{\pi}
G_0(x,\xi)\psi_{0}(\xi)d\xi=\mu_0\psi_0(x),
\end{eqnarray}
so $\psi_0(x)$  is also an eigenfunction of BVTP
(\ref{1.1})-(\ref{1.5}) corresponding to eigenvalue
$-\frac{1}{\mu_{0}}$. Since it is orthogonal to all eigenfunctions,
it is  orthogonal to itself, i.e.
\begin{eqnarray}\label{ex16} \frac{\Delta_{34}}{p_1}\int_{-\pi}^{0}\psi^{2}_0(s)ds+
\frac{\Delta_{12}}{p_2}\int_{0}^{\pi} \psi^{2}_0(s)ds=0
\end{eqnarray}
and hence $\psi_{0}(s)$ is identically zero. We get a contradiction,
which complete the proof.
\end{proof}

\section{Completeness of the eigenfunctions }
To prove the completeness $(\varphi_n)$ in square integrable
function space at first we shall prove the next theorem.
\begin{thm}\label{4.1}
Let $f(x)$ be any function on $[-\pi,0)\cup(0,\pi]$ satisfying the
following conditions\\
\textbf{i)}$f, f'$ and  $f''$ are exist and continuous in both
interval $[-\pi,0) \ and \ (0,\pi]$\\
\textbf{ii)}There exist a finite one-hand side limits
$f(\pm0), f'(\pm0) \  and  \ f''(\pm0)$\\
\textbf{iii)} f satisfies the BTC's (\ref{1.2})-(\ref{1.5}).\\ Then
f(x) can be expanded into an absolutely and uniformly convergent
series of eigenfunctions $(\varphi_n)$, namely
\begin{eqnarray}\label{com} f(x)=\sum_{n=0}^{\infty}\{
\frac{\Delta_{34}}{p_1}\int_{-\pi}^{0}f(\xi)\varphi_{n}(\xi)d\xi+
\frac{\Delta_{12} }{p_2}\int_{0}^{\pi}
f(\xi)\varphi_{n}(\xi)d\xi\}\varphi_{n}(x)
\end{eqnarray}
\end{thm}

\begin{proof}
Let f be satisfied all conditions of Theorem. Then by virtue of
(\ref{ex5}) the equality
\begin{eqnarray}\label{com1}
f(x)&=&\frac{\Delta_{34}}{p_1}\int_{-\pi}^{0}
G_0(x,\xi)(f''(\xi)-q(\xi)f(\xi))d\xi\nonumber\\&+&
\frac{\Delta_{12} }{p_2} \int_{0}^{\pi}
G(x,\xi)(f''(\xi)-q(\xi)f(\xi))d\xi
\end{eqnarray}
is hold. Substituting (\ref{ex10}) in the right hand of this
equality and integrating by parts twice yields the needed equality
\begin{eqnarray}\label{gr15} f(x)=\sum_{n=0}^{\infty}\{
\frac{\Delta_{34}}{p_1}\int_{-\pi}^{0}f(\xi)\varphi_{n}(\xi)d\xi+
\frac{\Delta_{12} }{p_2} \int_{0}^{\pi}
f(\xi)\varphi_{n}(\xi)dx\}\varphi_{n}(x)
\end{eqnarray}
The proof is complete.
\end{proof}
\begin{cor}
Let f be as in previous Theorem. Then the modified parseval equality
\begin{eqnarray}\label{gr17}
\frac{\Delta_{34}}{p_1}\int_{-\pi}^{0}f^{2}(x)dx+
\frac{\Delta_{12}}{p_2} \int_{0}^{\pi}
f^{2}(x)dx&=&\sum_{n=0}^{\infty}(\frac{\Delta_{34}}{p_1}\int_{-\pi}^{0}f(\xi)\varphi_{n}(\xi)d\xi\nonumber\\&+&
\frac{\Delta_{12}}{p_2} \int_{0}^{\pi}
f(\xi)\varphi_{n}(\xi)dx\}\varphi_{n}(\xi))^{2}
\end{eqnarray}
is hold. This is also called the completeness relation.
\end{cor}
We already know that the resolvent $u(x,\lambda)$ which determined
by (\ref{gr1}) for all $\lambda$ are not eigenvalues,  satisfies all
conditions of the Theorem. By using this fact, denoting $\ell
u(x)=p(x)u''(x)-q(x)u(x)$ and integrating by parts twice we have
\begin{eqnarray}\label{for}
c_n(\ell u(.,\lambda))&=&\frac{\Delta_{34}}{p_1}\int_{-\pi}^{0} \ell
u(\xi,\lambda)\varphi_{n}(\xi)d\xi+\frac{ \Delta_{12}}{p_2}
\int_{0}^{\pi} \ell u(\xi,\lambda)\varphi_{n}(\xi)d\xi\nonumber\\&=&
\frac{\Delta_{34}}{p_1}\int_{-\pi}^{0}
u(\xi,\lambda)\ell\varphi_{n}(\xi)d\xi+ \frac{\Delta_{12}}{p_2}
\int_{0}^{\pi}
u(\xi,\lambda)\ell\varphi_{n}(\xi)d\xi\nonumber\\&=&-\lambda_n
c_n(u(.,\lambda))
\end{eqnarray}
where the Fourier coefficients $c_n$ are defined as
\begin{eqnarray}\label{for1} c_n(f)=\sum_{n=0}^{\infty}
(\frac{\Delta_{34}}{p_1}\int_{-\pi}^{0}f(\xi)\varphi_{n}(\xi)d\xi+
\frac{\Delta_{12}}{p_2} \int_{0}^{\pi} f(\xi)\varphi_{n}(\xi)d\xi).
\end{eqnarray}
Since the resolvent $u(x,\lambda)$ satisfies the equation  $\ell
u(x,\lambda)+\lambda u(x,\lambda)=f(x)$ we have
\begin{eqnarray}\label{for1} c_n(f)&=&c_n(\ell
u(.,\lambda)+\lambda u(.,\lambda))=c_n (\ell u(.,\lambda))+\lambda
c_n (u(.,\lambda))\nonumber\\&=&-\lambda_nc_n (u(.,\lambda))+\lambda
c_n (u(.,\lambda))=(\lambda-\lambda_n)c_n (u(.,\lambda)) .
\end{eqnarray}

Thus if we know the expansion of given function f(x) into an
eigenfunctions, then the solution of homogeneous equation
(\ref{gr2}) satisfying BTC's (\ref{1.2})-(\ref{1.4}) can be obtain
by formula

\begin{eqnarray}\label{com} u(x,\lambda)=\sum_{n=0}^{\infty}\frac{1}{\lambda-\lambda_n}(
\frac{\Delta_{34}}{p_1}\int_{-\pi}^{0}f(\xi)\varphi_{n}(\xi)d\xi+
\frac{\Delta_{12}}{p_2} \int_{0}^{\pi}
f(\xi)\varphi_{n}(\xi)d\xi)\varphi_{n}(x))
\end{eqnarray}
for all $\lambda$ are not eigenvalues.
 \section{Expansion of mean-square integrable\\ functions into a series of eigenfunction}
The expansion theorem will now be extended to the square integrable
functions.
\begin{thm}
Let $f(x)$ be any square integrable function on $[-\pi,0) \ and \
(0,\pi]$. Then f can be expanded into Fourier series of
eigenfunctions in the sense of mean-square convergeness, namely the
formula
\begin{eqnarray}\label{me}
\lim_{n\rightarrow\infty}
\{\frac{\Delta_{34}}{p_1}\int\limits_{-\pi}^{0}(f(x)-\sum_{k=0}^{n}c_k(f)\varphi_k(x))^{2}dx+
\frac{\Delta_{12}}{p_2}\int\limits_{0}^{\pi}(f(x)-\sum_{k=0}^{n}c_k(f)\varphi_k(x))^{2}dx\}=0
\end{eqnarray}
is hold.
\end{thm}
\begin{proof}
Given any $\epsilon>0$, there exist an infinitely differentiable
function g(x) which vanish in the neighborhoods $x=-\pi, \  x=0 \
\textrm{and} \  x=\pi$ such that
\begin{eqnarray}\label{me1}
\frac{\Delta_{34}}{p_1}\int_{-\pi}^{0}(f(x)-g(x))^{2}dx+
\frac{\Delta_{12}}{p_2}\int_{0}^{\pi}(f(x)-g(x))^{2}dx<\epsilon
\end{eqnarray}
Denote, for shorting,
$$F_n(x)=\sum_{k=0}^{n}c_k(f)\varphi_k(x) \ \ \textrm{and} \ \ G_n(x)=\sum_{k=0}^{n}c_k(g)\varphi_k(x)$$
By virtue of Theorem (\ref{4.1})  there exist an integer $N_1$,
depending on $\epsilon$, such that
\begin{eqnarray}\label{me2}
\frac{\Delta_{34}}{p_1}\int_{-\pi}^{0}(g(x)-G_n(x))^{2}dx+
\frac{\Delta_{12}}{p_2}\int_{0}^{\pi}(g(x)-G_n(x))^{2}dx<\epsilon
\end{eqnarray}
for all $n\geq N_1$. By the well known Bessel inequality it can be
shown easily that
\begin{eqnarray}\label{me3}
\frac{\Delta_{34}}{p_1}\int_{-\pi}^{0}(G_n(x)-F_n(x))^{2}dx+
\frac{\Delta_{12}}{p_2}\int_{0}^{\pi}(G_n(x)-F_n(x))^{2}dx<\epsilon
\end{eqnarray}
for all $\varepsilon>0$. Finally, writing $f(x)-F_n(x)$ in the form
$f(x)-F_n(x)=(f(x)-g(x))+(g(x)-G_n(x))+(G_n(x)-F_n(x))$ and using
well-known Minkowski inequality from (\ref{me1}),  (\ref{me2}) and
(\ref{me3}) we can derive that
\begin{eqnarray}\label{me4}
\frac{\Delta_{34}}{p_1}\int_{-\pi}^{0}(f(x)-F_n(x))^{2}dx+
\frac{\Delta_{12}}{p_2}\int_{0}^{\pi}(f(x)-F_n(x))^{2}dx<3\epsilon
\end{eqnarray}
for  $n\geq N_1$ which proving the formula (\ref{me}).
\end{proof}
\begin{cor}
If f is as in previous theorem then the modified Parseval equality
\begin{eqnarray}\label{me5}
\frac{\Delta_{34}}{p_1}\int_{-\pi}^{0}f^{2}(x)dx+
\frac{\Delta_{12}}{p_2} \int_{0}^{\pi} f^{2}(x)dx
=\sum_{n=0}^{\infty}c^{2}_{n}(f).
\end{eqnarray}
is hold.
\end{cor}
This is also called the completeness relation.

\end{document}